\pdfoutput=1
\RequirePackage{ifpdf}
\ifpdf 
\documentclass[pdftex]{sigma}
\else
\documentclass{sigma}
\fi

\usepackage{mathdots}
\usepackage{extarrows}

\DeclareMathOperator{\ucp}{UCP}
\DeclareMathOperator{\rk}{rank}
\DeclareMathOperator{\toep}{Toep}
\DeclareMathOperator{\gauge}{\mathcal{G}}
\DeclareMathOperator{\pert}{Pert}
\DeclareMathOperator{\hu}{UCBH}
\DeclareMathOperator{\id}{Id}
\DeclareMathOperator{\vct}{vec}
\DeclareMathOperator{\Tr}{Tr}
\DeclareMathOperator{\cb}{CB}
\DeclareMathOperator{\image}{Image}
\DeclareMathOperator{\alg}{alg}
\DeclareMathOperator{\spn}{span}

\numberwithin{equation}{section}

\newtheorem{Theorem}{Theorem}[section]
\newtheorem*{Theorem*}{Theorem}
\newtheorem{Corollary}[Theorem]{Corollary}
\newtheorem{Lemma}[Theorem]{Lemma}
\newtheorem{Proposition}[Theorem]{Proposition}
 { \theoremstyle{definition}
\newtheorem{Definition}[Theorem]{Definition}

\newtheorem{Example}[Theorem]{Example}
\newtheorem{Remark}[Theorem]{Remark} }

\begin{document}
\allowdisplaybreaks

\newcommand{\arXivNumber}{2111.13076}

\renewcommand{\PaperNumber}{060}

\FirstPageHeading

\ShortArticleName{The Gauge Group and Perturbation Semigroup of an Operator System}

\ArticleName{The Gauge Group and Perturbation Semigroup\\ of an Operator System}

\Author{Rui DONG}

\AuthorNameForHeading{R.~Dong}

\Address{Institute for Mathematics, Astrophysics and Particle Physics, Radboud University Nijmegen,\\ Heyendaalseweg 135, 6525 AJ Nijmegen, The Netherlands}
\Email{\href{mailto:rui.dong@math.ru.nl}{rui.dong@math.ru.nl}}

\ArticleDates{Received December 01, 2021, in final form July 28, 2022; Published online August 09, 2022}

\Abstract{The perturbation semigroup was first defined in the case of $*$-algebras by Chamseddine, Connes and van Suijlekom. In this paper, we take $\mathcal{E}$ as a concrete operator system with unit. We first give a definition of gauge group $\mathcal{G}(\mathcal{E})$ of $\mathcal{E}$, after that we give the definition of perturbation semigroup of $\mathcal{E}$, and the closed perturbation semigroup of $\mathcal{E}$ with respect to the Haagerup tensor norm. We also show that there is a continuous semigroup homomorphism from the closed perturbation semigroup to the collection of unital completely bounded Hermitian maps over $\mathcal{E}$. Finally we compute the gauge group and perturbation semigroup of the Toeplitz system as an example.}


\Classification{46L07; 47L25; 58B34; 11M55}

\section{Introduction}

An operator system $\mathcal{E}$ is a matrix-normed vector space equipped with a conjugate linear map $x\mapsto x^*$ on $\mathcal{E}$ such that $(x^*)^*=x$ for all $x\in \mathcal{E}$.
Although there is not a well-defined product of elements in $\mathcal{E}$,
we can embed the operator system $\mathcal{E}$ into some $C^*$-algebra $\mathcal{A}$,
and then take the gauge group of $\mathcal{E}$ as the collection of unitary elements of $\mathcal{A}$ that keep $\mathcal{E}$ invariant under the unitary transformation,
i.e.,
\[
\mathcal{G}(\mathcal{E}):=\{u\in \mathcal{A}\colon u^*\mathcal{E}u=\mathcal{E}\}.
\]
There are several different approaches to embed $\mathcal{E}$ into a $C^*$-algebra,
for instance,
we can embed~$\mathcal{E}$ into the $C^*$-envelope $C^*_{\rm en}(\mathcal{E})$,
the injective envelope $C^*_{\rm in}(\mathcal{E})$,
or simply the $C^*$-algebra $C^*(\mathcal{E})$ generated by $\mathcal{E}$ when $\mathcal{E}$ is a concrete operator system.
In this paper, we take $\mathcal{E}$ to be a concrete closed operator system with unit,
i.e.,
a closed linear subspace of bounded operators on some Hilbert space $\mathcal{H}$ with $\id \in \mathcal{E}\subset B(\mathcal{H})$,
and we embed $\mathcal{E}$ into $C^*(\mathcal{E})$.
In Section \ref{sec_gauge}, we show that there is a group homomorphism from $\mathcal{G}(\mathcal{E})$ to the set of unital completely positive maps on $\mathcal{E}$.
In Section \ref{sec_gauge_toep}, we show that the gauge group $\mathcal{G}(\toep_n)$ of Toeplitz system $\toep_n$ is independent of $n$,
and
\[
\mathcal{G}(\toep_n)\cong U(1)\times(U(1)\rtimes \mathbb{Z}_2).
\]
Inspired by the definition of perturbation semigroup of $*$-algebras given in \cite{MR3090113},
 the perturbation semigroup of matrix algebras \cite{MR3500821}
and $C^*$-algebras \cite{laura},
in Section \ref{sec_pert}, we give the definition of the perturbation semigroup $\pert(\mathcal{E})$ of an operator system $\mathcal{E}$.
More than that,
since the perturbation semigroup $\pert(\mathcal{E})$ is a subset of $\mathcal{A}\otimes \mathcal{A}^\circ$,
we can take the closure of $\pert(\mathcal{E})$ with respect to the Haagerup tensor norm,
and we can show that there is a continuous semigroup homomorphism from this closure of $\pert(\mathcal{E})$ to the collection of unital completely bounded Hermitian maps on~$\mathcal{E}$.

In Section \ref{sec_pert_toep}, we discuss the perturbation semigroups $\pert(\toep_n)$
of Toeplitz system $\toep_n$ in more detail.
We show the relationship between an element $\omega\in\pert(\toep_n)$ and the corresponding $(2n-1)\times (2n-1)$ transformation matrix of Toeplitz system $\toep_n$ under the fundamental basis $\{\tau_{-n+1},\dots ,\tau_{0},\dots ,\tau_{n-1}\}$ of $\toep_n$.

\section{Gauge group of an operator system}\label{sec_gauge}
Let $\mathcal{H}$ be a separable Hilbert space,
we denote by $B(\mathcal{H})$ the set of bounded operators on $\mathcal{H}$,
$\mathcal{E}\subset B(\mathcal{H})$ an operator system,\footnote{Please check Appendix \ref{apd_os} for more details.}
and $C^*(\mathcal{E})$ the $C^*$-algebra generated by $\mathcal{E}$.
We are mainly interested in the unital completely positive (UCP) maps over $\mathcal{E}$.
According to Arveson's extension theorem \cite{MR253059, MR1976867},
if $\varphi\colon \mathcal{E}\to \mathcal{E}$ is a UCP map,
then there is a UCP map $\widetilde{\varphi}\colon B(\mathcal{H})\to B(\mathcal{H})$ such that $\widetilde{\varphi} \big|_{\mathcal{E}}=\varphi$.
In addition,
if $\widetilde{\varphi}$ is normal,\footnote{Please check Appendix \ref{apd_os} for the definition of normal map.}
according to Kraus \cite[Theorem~3.3 or Theorem~4.1]{MR292434},
the map~$\widetilde{\varphi}$ can be written as
\[
\widetilde{\varphi}(x)=\sum_k V_k^* x V_k, \qquad\forall x\in B(\mathcal{H}),
\]
for some operators $\{V_k\}_{k\in K}\subset B(\mathcal{H})$ such that $\sum V_k^* V_k=\id$.
Hence especially when $U\in C^*(\mathcal{E})$ is a unitary element satisfying $U^*\mathcal{E}U\subset \mathcal{E}$
the corresponding map $\varphi\colon x\mapsto U^*xU$ is a UCP map over $\mathcal{E}$.

{\sloppy
We denote by $\ucp(\mathcal{E})$ the collection of all the unital completely positive maps,
and $\ucp_{\rk=1}(\mathcal{E})$ the collection of rank-$1$ unital completely positive maps,
i.e.,
\[
\ucp_{\rk=1}(\mathcal{E}):=
\big\{\varphi\colon \mathcal{E}\to\mathcal{E}\mid \varphi(\cdot) = V^*(\cdot)V \textrm{ for some } V\in B(\mathcal{H}) \textrm{ with }V^*V=\id\big\}.
\]}\noindent
We realize that both $\ucp(\mathcal{E})$ and $\ucp_{\rk=1}(\mathcal{E})$ are semigroups with respect to the map composition.

\begin{Definition}
We define the gauge group $\gauge(\mathcal{E})$ of $\mathcal{E}$ as
\[
\gauge(\mathcal{E}):=\{U\in \mathcal{U}(C^*(\mathcal{E}))\mid U^*\mathcal{E}U= \mathcal{E}\},
\]
here $\mathcal{U}(C^*(\mathcal{E}))$ denotes the group of all the unitary elements in $C^*(\mathcal{E})$.
\end{Definition}

\begin{Remark}
If $\varphi(\cdot)=V^*(\cdot)V\in\ucp_{\rk=1}(\mathcal{E})$,
then $V\in B(\mathcal{H})$ is an isometry.
In particular,
if $\mathcal{E}\subset M_n(\mathbb{C})$ is a finite dimensional operator system,
then $V$ is a unitary matrix and $\ucp_{\rk=1}(\mathcal{E})$ is a group.
\end{Remark}

\begin{Proposition}
There is a multiplicative map $\Psi\colon \mathcal{G}(\mathcal{E})\to \ucp_{\rk=1}(\mathcal{E})$ defined as
\[
\Psi\colon\ U\mapsto U^*(\cdot)U,\quad U\in \mathcal{G}(\mathcal{E}).
\]
\end{Proposition}
We observe that the image of $\Psi$ forms a group and the map $\Psi\colon \mathcal{G}(\mathcal{E})\to \image(\Psi)$ is a~surjective group homomorphism.

\section{Perturbation semigroup of an operator system}\label{sec_pert}
In this section, we discuss unital completely bounded Hermitian($\hu$) maps and the perturbation semigroup of a concrete unital operator system $\mathcal{E}\subset B(\mathcal{H})$.

\begin{Definition}
We say $\Psi\colon \mathcal{E}\to\mathcal{E}$ is a Hermitian unital map if
$\Psi(x^*)=\Psi(x)^*$ for all $x\in \mathcal{E}$ and~$\Psi(\id)=\id$ for the unital element $\id\in \mathcal{E}$.
We denote by $\hu(\mathcal{E})$ the collection of all unital completely bounded Hermitian maps over $\mathcal{E}$,
i.e.,
\[
\hu(\mathcal{E}):=\big\{\Psi\colon \mathcal{E}\to \mathcal{E}\mid \Psi(x^*)=\Psi(x)^*,\, \Psi(\id)=\id,\, \Psi \textrm{ is completely bounded}\big\}.
\]
\end{Definition}

Inspired by the definition of perturbation semigroups introduced in \cite{MR3090113, laura, MR3500821},
we define the perturbation semigroup $\pert(\mathcal{E})$ of an operator system as follows:
\begin{Definition}\label{def_pert}
Let $\mathcal{E}$ be an operator system,
$C^*(\mathcal{E})$ be the $C^*$-algebra generated by~$\mathcal{E}$ and $C^*(\mathcal{E})^\circ$ be the opposite algebra of $C^*(\mathcal{E})$.
We define the perturbation semigroup $\pert(\mathcal{E})$ as the collection of all the finite sums of the form $\sum a_i\otimes b_i^\circ\in C^*(\mathcal{E})\otimes C^*(\mathcal{E})^\circ$ satisfying the following requirements:
\begin{enumerate}
\item[$1)$] $\sum a_i b_i=\id$,
\item[$2)$] $\sum a_i\mathcal{E} b_i\subset \mathcal{E}$,
\item[$3)$] $\sum a_i\otimes b_i^\circ = \sum b_i^*\otimes a_i^{*\circ}$.
\end{enumerate}
\end{Definition}

\begin{Remark}
In the definition above,
the opposite algebra $C^*(\mathcal{E})^\circ$ contains the same elements and addition operation as $C^*(\mathcal{E})$,
while the multiplication order is reversed.
And it is worth to observe that $(1)$ and $(3)$ inherit from the original definition of perturbation semigroup in \cite{MR3090113},
while $(2)$ is an extra condition we need to assume in our case of operator system.
\end{Remark}

For each $(a, b^\circ)\in C^*(\mathcal{E}) \times C^*(\mathcal{E})^\circ$,
let $\delta_{({a, b^\circ})}$ denote the completely bounded linear map on $C^*(\mathcal{E})$ in which $\delta_{({a, b^\circ})}(\xi)=a\xi b$,
for all $\xi\in C^*(\mathcal{E})$.
Let $\cb(C^*(\mathcal{E}))$ denote the set of all completed bounded maps over $C^*(\mathcal{E})$.
The map $C^*(\mathcal{E})\times C^*(\mathcal{E})^\circ\to \cb(C^*(\mathcal{E}))$ that sends each $(a, b^\circ)\in C^*(\mathcal{E}) \times C^*(\mathcal{E})^\circ$ to $\delta_{(a, b^\circ)}\in \cb(C^*(\mathcal{E}))$ is bilinear and therefore extends to a linear map $\Psi\colon C^*(\mathcal{E}) \otimes_{\alg} C^*(\mathcal{E})^\circ\to \cb(C^*(\mathcal{E}))$.

The perturbation semigroup $\pert(\mathcal{E})$ is a subset of $C^*(\mathcal{E}) \otimes_{\alg} C^*(\mathcal{E})^\circ$,
and so we define the map $\Phi\colon \pert(\mathcal{E})\to \cb(C^*(\mathcal{E}))$ by $\Phi=\Psi\big |_{\pert(\mathcal{E})}$.
Proposition \ref{prop_pt_hu} below shows that $\Phi$ is a~semigroup homomorphism of $\pert(\mathcal{E})$ into $\hu(\mathcal{E})$.

\begin{Proposition}\label{prop_pt_hu}
There is a semigroup homomorphism $\Phi$ from $\pert(\mathcal{E})$ to $\hu(\mathcal{E})$ defined~by
\[
\begin{aligned}
\Phi\colon\ \pert(\mathcal{E})&\to \hu(\mathcal{E}),
\\
\omega&\mapsto\sum a_i (\cdot) b_i
\end{aligned}
\]
with $\omega=\sum a_i\otimes b_i^\circ\in \pert(\mathcal{E})$.
\end{Proposition}

\begin{proof}
According to the definition of $\pert(\mathcal{E})$ any element $\omega\in \pert(\mathcal{E})$ can be written as $\omega=\sum a_i\otimes b_i^\circ=\sum b_i^*\otimes a_i^{*\circ}$,
thus $\Phi(\omega)$ is a Hermitian map.
The assumption that $\sum a_ib_i=\id$ confirms $\Phi(\omega)$ is unital.
Since there are only finitely many terms in the expression of the sum
\[
\Phi(\omega)\colon\ x \mapsto\sum a_i x b_i,
\qquad \forall x\in \mathcal{E},
\]
hence it is completely bounded due to \cite[Chapter 8]{MR1976867}.

Finally we shall show that the map $\Phi\colon \pert(\mathcal{E})\to \hu(\mathcal{E})$ is a semigroup homomorphism.
Let $\omega=\sum a_i\otimes b_i^\circ$ and $\widetilde{\omega}=\sum \widetilde{a}_j \otimes \widetilde{b}_j^\circ$ be two elements in $\pert(\mathcal{E})$,
we have that $\omega\widetilde{\omega}=\sum a_i\widetilde{a}_j\otimes\big(\widetilde{b}_jb_i\big)^\circ$,
and by Definition \ref{def_pert}
\[
\Phi(\omega\widetilde{\omega})(x)=
\sum a_i\widetilde{a}_j\,x\,\widetilde{b}_jb_i=
\sum_ia_i\bigg(\sum_{j}\widetilde{a}_j\,x\,\widetilde{b}_j\bigg)b_i\qquad
\textrm{for any}\quad x\in \mathcal{E},
\]
thus $\Phi(\omega\widetilde{\omega})=\Phi(\omega)\Phi(\widetilde{\omega})$ for $\omega, \widetilde{\omega}\in \pert(\mathcal{E})$.
\end{proof}

We can move one step further by equipping the semigroup $\pert(\mathcal{E})$ with the Haagerup tensor norm so that $\Phi$ can be extended to the closure of $\pert(\mathcal{E})$.
Recall that the Haagerup tensor norm\footnote{Please see Appendix \ref{apd_haag} for more details.} $\|u\|_h$ of an element $u\in C^*(\mathcal{E})\otimes C^*(\mathcal{E})^\circ$ is defined as
\[
\|u\|_h
=
\inf\bigg\{\Big\|\sum a_ia_i^*\Big\|^{1/2}\Big\|\sum b_i^* b_i\Big\|^{1/2}\bigg\},
\]
where the infimum is taken over all the expressions of $u=\sum a_i\otimes b_i^\circ$ for $a_i, b_i\in C^*(\mathcal{E})$.
Here we omit the opposite algebra structure.
Since $\pert(\mathcal{E})$ is a subset of $C^*(\mathcal{E})\otimes C^*(\mathcal{E})^\circ$,
we can endow $\pert(\mathcal{E})$ with the metric topology induced by the Haagerup tensor norm $\|\cdot \|_h$.

\begin{Definition}
We define the closed perturbation semigroup $\overline{\pert(\mathcal{E})}$ as the closure of $\pert(\mathcal{E})$ with respect to the topology induced by Haagerup tensor norm $\|\cdot\|_h$.
\end{Definition}

\begin{Proposition}\label{prop_cbhu}
Let $\mathcal{E}\subset B(\mathcal{H})$ be a unital operator system,
the map $\Phi\colon \pert(\mathcal{E})\to \hu(\mathcal{E})$ can be extended to a map
\[
\widetilde{\Phi}\colon\ \overline{\pert(\mathcal{E})}\to \hu(\mathcal{E}),
\]
such that $\widetilde{\Phi}\big |_{\pert(\mathcal{E})}=\Phi$.
Moreover,
if we equip $\overline{\pert(\mathcal{E})}$ and $\hu(\mathcal{E})$ with the metric topology induced by Haagerup tensor norm $\|\cdot\|_h$ and complete bound norm $\|\cdot\|_{cb}$ respectively,
the map~$\widetilde{\Phi}$ is contractive.
\end{Proposition}

\begin{proof}
By Definition \ref{def_pert} $\pert(\mathcal{E})$ is a subset of $C^*(\mathcal{E})\otimes_{\alg} C^*(\mathcal{E})$.
Take an element $\omega=\sum a_i \allowbreak\otimes b_i^\circ\in \pert(\mathcal{E})$,
we define a map $\widetilde{\Phi}\colon \pert(\mathcal{E})\to \cb(B(\mathcal{H}))$ as $\widetilde{\Phi}(\omega)\colon T\mapsto \sum a_i T b_i$ for $T\in B(\mathcal{H})$.
According to \cite[Theorem 5.12]{MR2006539},
the map $\widetilde{\Phi}$ is completely isometric if we equip with $\omega$ the Haagerup norm and $\widetilde{\Phi}(\omega)$ the completely bounded norm.
If we can take the closure $\overline{\pert(\mathcal{E})}$,
we get a map from $\overline{{\pert(\mathcal{E})}}$ to $\cb(B(\mathcal{H}))$,
which we still denote as $\widetilde{\Phi}$.
By our definition of $\widetilde{\Phi}$, we observe that $\widetilde{\Phi}\big |_{\pert(\mathcal{E})}=\Phi$,
hence we only need to show that the image of $\widetilde{\Phi}$ is contained in~$\hu(\mathcal{E})$.

Take a sequence of $\{\omega_n\}_{n\geq 1}\subset \pert(\mathcal{E})$ that approaches to some $\omega\in \overline{\pert(\mathcal{E})}$.
Since
\[
\widetilde{\Phi}(\omega_n)(\id)=\Phi(\omega_n)(\id)=\id,
\]
we obtain that $\widetilde{\Phi}(\omega)$ is a unital map.
Similarly,
since for each $\omega_n$ the map $\Phi(\omega_n)$ is Hermitian,
we conclude that $\widetilde{\Phi}(\omega)$ is Hermitian.
Hence we only need to show that for any $x\in \mathcal{E}$,
$\widetilde{\Phi}(\omega)(x)\in \mathcal{E}$.

In fact,
for any $\epsilon>0$,
there exists an $N>0$ such that when $n\geq N$ we have $\|\omega_n-\omega\|_h
< \epsilon$.
Besides that,
according to \cite[Theorem 5.12]{MR2006539},
if we regard $\widetilde{\Phi}(\omega_n)-\widetilde{\Phi}(\omega)$ as a map on $B(\mathcal{H})$ we can obtain that
$\big\|\widetilde{\Phi}(\omega_n)-\widetilde{\Phi}(\omega)\big\|_{cb}
=
\|\omega_n-\omega\|_h$,
since $\mathcal{E}\subset B(\mathcal{H})$.
For the restriction of $\widetilde{\Phi}(\omega_n)-\widetilde{\Phi}(\omega)$ to $\mathcal{E}$ we obtain $\big\|\widetilde{\Phi}(\omega_n)-\widetilde{\Phi}(\omega)\big\|_{cb}
\leq
\|\omega_m-\omega\|_h$.
Hence
\[
\big\|\widetilde{\Phi}(\omega_n)-\widetilde{\Phi}(\omega)\big\|
\leq
\big\|\widetilde{\Phi}(\omega_n)-\widetilde{\Phi}(\omega)\big\|_{cb}
\leq
\|\omega_n-\omega\|_h
< \epsilon.
\]
Thus if we take an $x\in\mathcal{E}$,
we have
\[
\frac{\big\|\widetilde{\Phi}(\omega_n)(x)-\widetilde{\Phi}(\omega)(x)\big\|}{\|x\|}
<
\epsilon.
\]
Therefore $\widetilde{\Phi}(\omega_n)(x)\to \widetilde{\Phi}(\omega)(x)$.
So that by closedness of $\mathcal{E}$ we obtain that $\widetilde{\Phi}(\omega)(x)\in \mathcal{E}$.

Hence for an element $\omega\in \overline{\pert(\mathcal{E})}$,
we can consider $\widetilde{\Phi}(\omega)$ as either an element of $\hu(B(\mathcal{H}))$ or an element of $\hu(\mathcal{E})$.
However,
since $\mathcal{E}\subset B(\mathcal{H})$ is a subset,
if we regard $\widetilde{\Phi}(\omega)$ as a element in $\hu(\mathcal{E})$,
the completely bounded norm of $\widetilde{\Phi}(\omega)$ is less than or equal to the completely bounded norm of $\widetilde{\Phi}(\omega)$ as an element of $\hu(B(\mathcal{H}))$.
Therefore the map $\widetilde{\Phi}$ is contractive.
\end{proof}

For a general operator system $\mathcal{E}$ we can only conclude the map $\widetilde{\Phi}\colon \overline{\pert(\mathcal{E})}\to \hu(\mathcal{E})$ is completely contractive rather than completely isometric.

\begin{Example}
Let $\{E_{ij}\}$, $1\leq i, j \leq 2$ be the standard matrix units for $M_2(\mathbb{C})$.
Define
\[
\toep_2=\left\{\begin{pmatrix}
a & b\\c & a
\end{pmatrix}
\subset M_2(\mathbb{C})\right\}\!.
\]
Take $\omega_1, \omega_2\in \pert(\toep_2)$ given as
\begin{gather*}
\omega_1=
E_{12}\otimes E_{12}^{\,\circ} + E_{21}\otimes E_{21}^{\,\circ} + E_{11}\otimes E_{11}^{\,\circ} + E_{22}\otimes E_{22}^{\,\circ},
\\
\omega_2=(E_{12}+E_{21})\otimes (E_{12}+E_{21})^\circ.
\end{gather*}
By a direct computation we obtain that $\Phi(\omega_1)=\Phi(\omega_2)$ on $\toep_2$,
both give rise to the transposition map on $\toep_2$,
and we observe that $E_{12}+E_{21}$ is a $2\times 2$ unitary matrix,
thus $\|\Phi(\omega_2)\|_{cb}=1$,
and therefore we obtain that $\|\Phi(\omega_1)\|_{cb}=1$.

However,
according to \cite[Theorem 17.4]{MR1976867},
the Haagerup tensor norm $\|\omega_1\|_h$ is equal to the completely bounded norm of the transposition transformation over $M_2(\mathbb{C})$,
which is equal to $2$.
Therefore,
$\|\Phi(\omega_1)\|_{cb}=1<\|\omega_1\|_h=2$.
\end{Example}

\begin{Definition}
We denote by $\pert^+(\mathcal{E})$ the sub-semigroup of $\pert(\mathcal{E})$ containing all the $\omega\in\pert(\mathcal{E})$ of the form $\omega=\sum a_i\otimes a_i^{*\circ}$ for some $a_i\in C^*(\mathcal{E})$,
i.e.,
\[
\pert^+(\mathcal{E}):=\Big\{\omega\in\pert(\mathcal{E})\,\Big |\, \omega=\sum a_i\otimes a_i^{*\circ}\textrm{ for some }a_i\in C^*(\mathcal{E})\Big\}.
\]
\end{Definition}
To simplify the notation we still denote the restriction $\Phi|_{\pert^+(\mathcal{E})}$ to $\pert^+(\mathcal{E})$ by $\Phi$.

\begin{Corollary}
Let $\omega=\sum a_i\otimes a_i^{*\circ}\in \pert^+(\mathcal{E})$.
We have $\Phi(\omega)\in \ucp(\mathcal{E})$,
namely
\[
\begin{aligned}
\Phi\colon \ \pert^+(\mathcal{E})&\to\ucp(\mathcal{E}),
\\
\omega&\mapsto\sum a_i(\cdot)a_i^*.
\end{aligned}
\]
\end{Corollary}

\begin{proof}
By Proposition \ref{prop_pt_hu} we have that $\Phi(\omega)\in \hu(\mathcal{E})$ for $\omega\in \pert^+(\mathcal{E})$,
and $\Phi(\omega)(\cdot)=\sum a_i(\cdot) a_i^*$,
which is a completely positive map.
\end{proof}

As in the case of $\pert(\mathcal{E})$,
we can take the closure of $\pert^+(\mathcal{E})$ with respect to Haagerup tensor norm, which we denote as $\overline{\pert^+(\mathcal{E})}$.

\begin{Proposition}\label{prop_ucp}
Let $\overline{\pert^+(\mathcal{E})}$ be the closure of $\pert^+(\mathcal{E})$ with respect to Haagerup tensor norm.
We can extend the map $\Phi\colon \pert^+(\mathcal{E})\to \ucp(\mathcal{E})$ to a map
\[
\widetilde{\Phi}\colon \ \overline{\pert^+(\mathcal{E})}\to \ucp(\mathcal{E}),
\]
such that $\widetilde{\Phi}\big |_{\pert^+(\mathcal{E})}=\Phi$.
Moreover,
we have $\|\omega\|_h=1$ and $\|\widetilde{\Phi}(\omega)\|_{cb}=1$ for every $\omega\in \overline{\pert^+(\mathcal{E})}$.
\end{Proposition}

\begin{proof}
Take an element $\omega\in \overline{\pert^+(\mathcal{E})}$,
according to Proposition \ref{prop_cbhu},
the map $\widetilde{\Phi}(\omega)\in \hu(\mathcal{E})$.
we then need to show that $\widetilde{\Phi}(\omega)$ is completely positive.
Indeed,
if we take a sequence $\{\omega_n\}_{n\geq 1}\subset \pert^+(\mathcal{E})$ such that $\omega_n\to \omega$,
then for any $\epsilon>0$,
there exists an $N>0$ such that when $n\geq N$
\begin{equation}\label{ineq_ucp}
\big\|\widetilde{\Phi}(\omega_n)-\widetilde{\Phi}(\omega)\big\|_{cb}
\leq
\|\omega_n-\omega\|_h
< \epsilon.
\end{equation}
Take a positive element $X_k\in M_k(\mathcal{E})$,
then $\widetilde{\Phi}(\omega_n)(X_k)\in M_k(\mathcal{E})$ is also positive for all $n\in \mathbb{N}$.
And by the inequality \eqref{ineq_ucp},
we have
\[
\frac{\big\|\widetilde{\Phi}(\omega_n)(X_k)-\widetilde{\Phi}(\omega)(X_k)\big\|}{\|X_k\|}< \epsilon,
\]
that is to say,
$\widetilde{\Phi}(\omega)(X_k)$ is the limit point of the sequence of positive elements $\big\{\widetilde{\Phi}(\omega_n)(X_k)\big\}_{n\geq 1}$ in $M_k(\mathcal{E})$,
thus $\widetilde{\Phi}(\omega)(X_k)\in M_k(\mathcal{E})$ is positive.
Since this is true for all $k\in \mathbb{N}$,
$\widetilde{\Phi}(\omega)$ is completely positive and therefore $\widetilde{\Phi}(\omega)\in \ucp(\mathcal{E})$.

Finally, we only need to show that $\big\|\widetilde{\Phi}(\omega)\big\|_{cb}=\|\omega\|_{h}=1$ for each $\omega\in \overline{\pert^+(\mathcal{E})}$.
Take an element $\omega\in \overline{\pert^+(\mathcal{E})}$.
For any $\epsilon>0$,
there exists an $\omega^\prime\in \pert^+(\mathcal{E})$ such that
\[
\|\Phi(\omega^\prime)\|_{cb} - \epsilon
\leq \big\|\widetilde{\Phi}(\omega)\big\|_{cb}
\leq \|\omega\|_h
\leq \|\omega^\prime\|_h+ \epsilon.
\]
Since $\omega^\prime\in\pert^+(\mathcal{E})$,
we can write $\omega^\prime$ as $\omega^\prime=\sum\limits_{i=1}^k a_i\otimes a_i^{*\circ}$ for some $a_i\in C^*(\mathcal{E})$,
and according to Definition \ref{def_pert}, we obtain that $\sum\limits_{i=1}^ka_i a_i^*=\id$.
Thus
\begin{gather*}
\|\omega^\prime\|_h\leq \Bigg\|\sum_{i=1}^k a_i a_i^*\Bigg\|= 1.
\end{gather*}
On the other hand,
we observe the inequality
\[
\|\Phi(\omega^\prime)\|_{cb}
\geq \|\Phi(\omega^\prime)\|
\geq \|\Phi(\omega^\prime)(\id)\|
= 1.
\]
Hence combine the three inequalities above together we conclude that
\[
1 - \epsilon\leq \big\|\widetilde{\Phi}(\omega)\big\|_{cb}\leq \|\omega\|_h \leq 1 + \epsilon.
\]
Since this is true for every $\epsilon>0$,
we obtain that $\big\|\widetilde{\Phi}(\omega)\big\|_{cb}=\|\omega\|_h=1$ for all $\omega\in \overline{\pert^{+}(\mathcal{E})}$.
\end{proof}

We also observe that there is a map from the gauge group $\mathcal{G}(\mathcal{E})$ to the semigroup $\pert^+(\mathcal{E})$,
as stated in the following proposition.

\begin{Proposition}
There is a multiplicative map from $\mathcal{G}(\mathcal{E})$ to $\pert^+(\mathcal{E})$ defined by
\begin{equation*}
\begin{aligned}
\mathcal{G}(\mathcal{E})
&\to\pert^+(\mathcal{E}),
\\
u&\to u^*\otimes u^{\circ}.
\end{aligned}
\end{equation*}
\end{Proposition}

\begin{Remark}
Although for an element $\omega\in \overline{\pert^+(\mathcal{E})}$ we have $\big\|\widetilde{\Phi}(\omega)\big\|_{cb}
=
\|\omega\|_{h}
=1,
$
the completely bounded norm $\big\|\widetilde{\Phi}(\omega_1)-\widetilde{\Phi}(\omega_2)\big\|_{cb}$ and the Haagerup norm of $\|\omega_1-\omega_2\|_h$ for two elements $\omega_1, \omega_2\in \overline{\pert^+(\mathcal{E})}$ usually are not equal.

Consider the $2\times 2$ Toeplitz system $\toep_2$.
Take $\omega_1, \omega_2\in \pert^+(\toep_2)$ as
\[
\omega_1=E_{11}\otimes E_{11}^\circ + E_{22}\otimes E_{22}^\circ,
\qquad
\omega_2=E_{12}\otimes E_{21}^\circ + E_{21}\otimes E_{12}^\circ,
\]
although $\omega_1\neq \omega_2$,
we have $\Phi(\omega_1)=\Phi(\omega_2)$.
Indeed,
\[
\Phi(\omega_1)=\Phi(\omega_2)\colon
\begin{pmatrix}
a & b\\
c & a\\
\end{pmatrix}
\mapsto
\begin{pmatrix}
a & 0\\
0 & a
\end{pmatrix}\!.
\]
Therefore $\big\|\widetilde{\Phi}(\omega_1)-\widetilde{\Phi}(\omega_2)\big\|_{cb}$ is equal to $0$ while $\|\omega_1-\omega_2\|_h$ is not.
\end{Remark}

\section[Gauge group and perturbation semigroup of the Toeplitz system]
{Gauge group and perturbation semigroup \\of the Toeplitz system}
The concept of truncated circle is introduced by Alain Connes and Walter D. van Suijlekom in~\cite{MR4244265}.
Recall that the canonical spectral triple on the circle is in the form of
\[
\bigg(C^\infty\big(S^1\big), \, L^2\big(S^1\big),\, D=-{\rm i}\frac{{\rm d}}{{\rm d}t}\bigg)
\]
as discussed in \cite[Chapter~5]{MR3237670}.
Let $\{e_n\}_{n\in \mathbb{Z}}$ be the set of eigenvectors of $D$,
we consider a~spectral truncation defined by the orthogonal projection $P_n$ onto $\spn_{\mathbb{C}}\{e_1, e_2, \dots , e_n\}$ for some $n>0$.
The truncated circle with respect to $P_n$ is defined as
\[
\big(P_nC^\infty\big(S^1\big)P_n, P_nL^2\big(S^1\big), P_nDP_n\big).
\]
Since $P_n$ does not commute with the $*$-algebra $C^\infty\big(S^1\big)$,
$P_nC^\infty\big(S^1\big)P_n$ is only an operator system rather than an algebra.
In fact,
if $f\in C^\infty\big(S^1\big)$ is a smooth function with Fourier coefficients $\{a_n\}_{n\in\mathbb{Z}}$,
then the truncation $P_n fP_n$ can be written as a Toeplitz matrix:
\[
P_n fP_n
=
\begin{pmatrix}
a_0 & a_{-1} & \cdots & a_{-n+2} & a_{-n+1}\\
a_1 & a_0 & a_{-1} & \cdots & a_{-n+2}\\
\vdots & a_1 & a_0 & \ddots & \vdots \\
a_{n-2} & \vdots & \ddots & \ddots & a_{-1}\\
a_{n-1} & a_{n-2} & \cdots & a_1 & a_0
\end{pmatrix}\!.
\]
Hence it turns out that $P_n C^\infty\big(S^1\big)P_n$ is the Toeplitz operator system containing all the $n\times n$ Toeplitz matrices,
which we denote as $\toep_n$.

One interesting question is what are the gauge group and perturbation semigroup of the Toeplitz operator system $\toep_n$.
In this section, we will present the structure of gauge group $\mathcal{G}(\toep_n)$ and some properties of perturbation semigroup $\pert(\toep_n)$.
Many properties of $\toep_n$ are different from that of $M_n(\mathbb{C})$,
in this section, we will also show that the transpose map on $\toep_n$ is a $\ucp$ map,
which is absolutely wrong in the case of $M_n(\mathbb{C})$.
The readers can refer to \cite{MR4340832} for more details and other interesting behaviors about Toeplitz operator system.

\subsection{Gauge group of the Toeplitz system}\label{sec_gauge_toep}

As is shown in \cite{MR4244265},
the $C^*$-algebra generated by $\toep_n$ is just $M_n(\mathbb{C})$.
The main goal of this section is to figure out $\mathcal{G}(\toep_n)$.
One interesting phenomenon is that $\mathcal{G}(\toep_n)$ is independent of $n$.
Before proving that we need the following lemma.
\begin{Lemma}\label{lem_1}
Let $U\in \gauge(\toep_n)$,
then $U$ is either a diagonal matrix or an anti-diagonal matrix.
\end{Lemma}
\begin{proof}
We take a unitary matrix $U=(u_{ij})_{1\leq i,j\leq n}\in
\mathcal{U}(M_n(\mathbb{C}))$ and a basis $\{\tau_j\}_{j=-n+1,\ldots, n-1}$ of the Toeplitz system $\toep_n$ given by $1$'s on the $j$'th diagonal and $0$'s elsewhere,
i.e.,
for positive~$k$ we have
\[
\tau_k=\sum_{i=1}^{n-k}E_{i, i+k},
\qquad
\tau_{-k}=\sum_{i=1}^{n-k}E_{i+k, i},
\]
here $E_{i,j}$ is the $n\times n$ unit matrix with $1$ in $(i, j)$-entry and $0$'s everywhere else.
An element $U\in \gauge(\toep_n)$ if and only if $U^*\tau_{j}U\in\toep_n$ for all $j\in[-n+1,n-1]$.
We observe first that when $k>0$ the $(j,l)$-entry of $U^*\tau_{k}U$ is given by
\begin{equation}\label{eq_element}
(U^*\tau_k U)_{j,l}=
\sum\limits_{i=1}^{n-k}\overline{u}_{i,j}u_{k+i,l},\qquad 1\leq j, l\leq n,
\end{equation}
and
\[
\Tr(U^* \tau_k U)=\sum_{j=1}^n\sum\limits_{i=1}^{n-k}\overline{u}_{i,j}u_{k+i,j}.
\]
Since $U$ is a unitary matrix,
we have $\sum\limits_{j=1}^n\overline{u}_{i,j}u_{k+i,j}=0$ for $k > 0$ and $1\leq i\leq n-k$.
Thus we have
\[
\Tr(U^*\tau_{k}U)=
\sum_{j=1}^n\sum\limits_{i=1}^{n-k}\overline{u}_{i,j}u_{k+i,j}=0,\qquad k> 0.
\]
Due to our assumption that $U^*\tau_kU\in\toep_n$,
we must have all the diagonal entries of $U^*\tau_{k}U$ are zeros:
\begin{equation}\label{eq_diag}
\sum\limits_{i=1}^{n-k}\overline{u}_{i,j}u_{k+i,j}=0, \qquad k>0, \quad 1\leq j\leq n.
\end{equation}

Take $k=n-1$ and $j=1$ in formula \eqref{eq_diag},
we have that $\overline{u}_{1,1}u_{n,1}=0$.
However $\overline{u}_{1,1}$ and $u_{n,1}$ can not be both equal to $0$,
otherwise by equation \eqref{eq_element},
$U^*\tau_{n-1}U=0$.
In fact,
the equation~\eqref{eq_element} implies that
$(U^*\tau_{n-1}U)_{j,l}=\overline{u}_{1,j}u_{n,l}$,
suppose if $\overline{u}_{1,1}=u_{n,1}=0$,
then $(U^*\tau_{n-1}U)_{1,l}=(U^*\tau_{n-1}U)_{j,1}=0$ for all $1\leq j,l\leq n$.
That is,
all the entries in the first row and the first column of $U^*\tau_{n-1}U$ are $0$'s,
since we assume the matrix $U^*\tau_{n-1}U$ is a Toeplitz matrix,
we conclude that $(U^*\tau_{n-1}U)=0$.
This is a contradiction of the unitary of $U$.

Since $u_{1,1}$ and $u_{n,1}$ can not both be equal to $0$, we first assume that $u_{1,1}=\alpha\neq 0$ and $u_{n,1}=0$.
If we take $k=n-2$ and $j=1$ in formula \eqref{eq_diag},
we obtain that
\begin{equation}\label{eq_diag_2}
\overline{u}_{1,1}u_{n-1,1}+\overline{u}_{2,1}u_{n,1}=0.
\end{equation}
Therefore we have $u_{n-1,1}=0$.
We then take $k=n-3$ and $j=1$ in formula \eqref{eq_diag} again,
we obtain the equation
\begin{equation}\label{eq_diag_3}
\overline{u}_{1,1}u_{n-2,1}+\overline{u}_{2,1}u_{n-1,1}+\overline{u}_{3,1}u_{n,1}=0,
\end{equation}
since $u_{1,1}\neq 0, u_{n,1}=0$ and $u_{n-1,1}=0$,
we obtain that $u_{n-2,1}=0$.
By induction,
take $j=1$ and $k=n-4, n-5,\ldots,2, 1$,
we obtain that $u_{i,1}=0$ for $1<i\leq n$,
namely,
all the entries in the first column of $U$ are equal to $0$ except that $u_{1,1}=\alpha\neq 0$.
Thus we can write $U$ in the matrix form as
\[
U=\begin{pmatrix}
\alpha & u_{1,2} & u_{1,3} & \cdots & u_{1,n}\\
0 & u_{2,2} & u_{2,3} & \cdots & u_{2,n}\\
0 & u_{3,2} & u_{3,3} & \cdots & u_{3,n}\\
\vdots & \vdots & \vdots & \ddots & \vdots\\
0 & u_{n,2} & u_{n,3} & \cdots & u_{n,n}
\end{pmatrix}\!,
\]
and by a simple computation
\begin{equation}\label{eq_mtx}
U^*\tau_{n-1}U=\begin{pmatrix}
0 &\overline{\alpha} u_{n,2} &\overline{\alpha} u_{n,3} & \cdots &\overline{\alpha} u_{n,n}\\
0 & \overline{u}_{1,2}u_{n,2} & \overline{u}_{1,2}u_{n,3} & \cdots & \overline{u}_{1,2}u_{n,n}\\
0 & \overline{u}_{1,3}u_{n,2} & \overline{u}_{1,3}u_{n,3} & \cdots & \overline{u}_{1,3} u_{n,n}\\
\vdots & \vdots & \vdots & \ddots & \vdots\\
0 & \overline{u}_{1,n}u_{n,2} & \overline{u}_{1,n}u_{n,3} & \cdots & \overline{u}_{1,n}u_{n,n}
\end{pmatrix}\!.
\end{equation}
Now we show that $u_{1,2}=u_{1,3}=\cdots =u_{1,n}=0$.
Since the matrix \eqref{eq_mtx} is a Toeplitz matrix,
the $(2, 2)$-entry element in \eqref{eq_mtx} must be equal to $0$,
namely $\overline{u}_{1,2}u_{n,2}=0$.
Suppose $u_{1,2}\neq 0$,
then we must have $u_{n,2}=0$,
which implies that the second column of \eqref{eq_mtx} is $0$.
It then implies that the $(2,3)$-entry element in \eqref{eq_mtx} is equal to $0$,
which implies $u_{n,3}=0$ and thus the third column of \eqref{eq_mtx} is $0$.
By induction,
we obtain that $u_{n,1}=u_{n,2}=u_{n,3}=\cdots =u_{n,n-1}=u_{n,n}=0$,
that is,
$U^*\tau_{n-1}U=0$,
which is impossible.
Therefore we must have $u_{1,2}=0$,
and we deduce that all the entries in the second row of \eqref{eq_mtx} are $0$'s.
Hence the only non-zero entry in \eqref{eq_mtx} is the $(1,n)$-entry and all the rest entries are $0$'s.
Namely,
\[
U^*\tau_{n-1}U=\overline{\alpha}u_{n,n}\tau_{n-1}.
\]
Thus we obtain that $u_{1,2}=u_{1,3}=\cdots =u_{1,n}=0$.
That is to say,
the unitary matrix $U$ is of the form
\[
U=\begin{pmatrix}
\alpha & 0\\
0 & \widetilde{U}
\end{pmatrix}\!,\qquad
|\alpha|=1,\quad \widetilde{U}\in \mathcal{U}(M_{n-1}(\mathbb{C})).
\]
Take a Toeplitz matrix $T=(t_{ij})_{1\leq i,j\leq n}\in \toep_{n}$,
we write $T$ in the block form as
\[
T=
\begin{pmatrix}
x & X\\
Y & \widetilde{T}
\end{pmatrix},
\qquad
\widetilde{T}\in \toep_{n-1},
\]
here $x=t_{11}$, $X=(t_{12}, \dots , t_{1n})$,
and $Y=(t_{21},\dots , t_{n1})^{\rm T}$.
A simple computation shows that
\[
U^{*}TU=
\begin{pmatrix}
x & \overline{\alpha} X \widetilde{U}\\
\alpha \widetilde{U}^* Y & \widetilde{U}^*\widetilde{T}\widetilde{U}
\end{pmatrix}\in \toep_{n},
\]
which implies that $\widetilde{U}^*\widetilde{T}\widetilde{U}\in\toep_{n-1}$.
Apply the same argument to $\widetilde{U}\in \toep_{n-1}$,
we obtain that the $(n-1)\times(n-1)$ unitary matrix $\widetilde{U}$ is of the form
\[
\widetilde{U}=\begin{pmatrix}
\beta & 0\\
0 & \widehat{U}
\end{pmatrix}\!,\qquad
|\beta|=1,\quad \widehat{U}\in \mathcal{U}(M_{n-2}(\mathbb{C})),
\]
apply the same argument to $\widehat{U}$ again,
by induction we obtain that $U$ is a diagonal matrix when $u_{1,1}\neq 0$.

On the other hand,
when $u_{1,1}=0$ and $u_{n,1}=\alpha\neq 0$,
the equation \eqref{eq_diag_2} then implies that $u_{2,1}=0$,
and the equation \eqref{eq_diag_3} implies that $u_{3,1}=0$,
by induction,
take $k=n-4, n-5,\ldots, 2,1$ and $j=1$,
we obtain that the first column of $U$ are all zeros except $u_{n,1}\neq 0$.
Namely the unitary matrix $U$ is of the form
\[
U=\begin{pmatrix}
0 & u_{1,2} & \cdots & u_{1,n-1} & u_{1,n}\\
0 & u_{2,2} & \cdots & u_{2,n-1} & u_{2,n}\\
\vdots & \vdots & \ddots & \vdots & \vdots\\
0 & u_{n-1,2} & \cdots & u_{n-1,n-1} & u_{n-1,n}\\
\alpha & u_{n,2} & \cdots & u_{n,n-1} & u_{n,n}
\end{pmatrix}\!,
\]
and by a direct computation we can write the matrix $U^*\tau_{n-1}U$ as
\[
U^*\tau_{n-1}U=\begin{pmatrix}
0 &0 & \cdots & 0 & 0\\
\overline{u}_{1,2}\alpha & \overline{u}_{1,2}u_{n,2} & \cdots & \overline{u}_{1,2}u_{n,n-1} & \overline{u}_{1,2}u_{n,n}\\
\vdots & \vdots & \ddots & \vdots & \vdots\\
\overline{u}_{1,n-1}\alpha & \overline{u}_{1,n-1}u_{n,2} & \cdots & \overline{u}_{1,n-1}u_{n,n-1} & \overline{u}_{1,n-1} u_{n,n}\\
\overline{u}_{1,n}\alpha & \overline{u}_{1,n}u_{n,2} & \cdots & \overline{u}_{1,n}u_{n,n-1} & \overline{u}_{1,n}u_{n,n}
\end{pmatrix}\!.
\]
Using a similar argument as in the case of $u_{1,1}\neq 0$ and $u_{n,1}=0$,
we can deduce that
$U$ is an anti-diagonal matrix if $u_{1,1}=0$ and $u_{n,1}\neq 0$.
\end{proof}

The gauge group $\gauge(\toep_n)$ has a more explicit expression as given below.
\begin{Proposition}
The gauge group $\gauge(\toep_n)$ is generated by the diagonal matrices $U_{\alpha, \beta}$ and anti-diagonal matrix $V$ of the form
\begin{gather}\label{eq_U_V}
U_{\alpha, \beta}=
\begin{pmatrix}
\alpha & 0 & 0 & \cdots & 0\\
0 & \beta & 0 & \cdots & 0\\
0 & 0 & \overline{\alpha}\beta^2 & \cdots & 0\\
\vdots & \vdots & \vdots & \ddots & \vdots\\
0 & 0 & 0 & \cdots & \overline{\alpha}^{n-2}\beta^{n-1}\\
\end{pmatrix}\!,\quad
V=\begin{pmatrix}
0 & \cdots & 0 & 0 & 1\\
0 & \cdots & 0 & 1 & 0\\
0 & \cdots & 1 & 0 & 0\\
\vdots & \iddots & \vdots & \vdots & \vdots\\
1 & \cdots & 0 & 0 & 0\\
\end{pmatrix}\!,\quad
|\alpha|=|\beta|=1.\!\!
\end{gather}
\end{Proposition}

\begin{proof}
According to Lemma \ref{lem_1},
any $U\in \mathcal{G}(\toep_n)$ is either a diagonal matrix or an anti-diagonal matrix.
Suppose $U$ is a diagonal matrix,
then $U$ can be expressed as
\[
U=
\begin{pmatrix}
\alpha_1 & 0 & \cdots & 0\\
0 & \alpha_2 & \cdots & 0\\
\vdots & \vdots & \ddots & \vdots\\
0 & 0 & \cdots & \alpha_n\\
\end{pmatrix}
\]
with $|\alpha_i|=1$ for $i=1,2,\ldots, n$.
We then obtain
\[
U^*\tau_{1}U
=
\begin{pmatrix}
0 & \overline{\alpha}_1\alpha_2 & 0 & \cdots & 0\\
0 & 0 & \overline{\alpha}_2\alpha_3 &\cdots & 0\\
\vdots & \vdots & \ddots & \cdots & \vdots\\
0 & 0 & 0 & \cdots & \overline{\alpha}_{n-1}\alpha_n\\
0 & 0 & 0 & \cdots & 0\\
\end{pmatrix}\!,
\]
since $U^*\tau_1U\in \toep_n$,
we must have $\overline{\alpha}_1\alpha_2=\overline{\alpha}_2\alpha_3=\cdots=\overline{\alpha}_{n-1}\alpha_n$.
If we take $\alpha_1=\alpha$ and $\alpha_2=\beta$,
we must have $\alpha_i=\overline{\alpha}^{i-2}\beta^{i-1}$ for $3\leq i\leq n$,
hence we obtain the unitary matrix $U_{\alpha,\beta}$ as given in \eqref{eq_U_V}.

Now suppose if the unitary matrix $W$ is an anti-diagonal matrix of the form
\[
W=
\begin{pmatrix}
0 & 0 & \cdots & 0 & \alpha_1\\
0 & 0 & \cdots & \alpha_2 & 0\\
\vdots & \vdots & \ddots & \vdots & \vdots\\
0 & \alpha_{n-1} & \cdots & 0 & 0\\
\alpha_n & 0 & \cdots & 0 & 0\\
\end{pmatrix}\!.
\]
Using a similar argument we can show that
\[
W=
\begin{pmatrix}
0 & 0 & \cdots & 0 & \alpha\\
0 & 0 & \cdots & \beta & 0\\
\vdots & \vdots & \ddots & \vdots & \vdots\\
0 & \overline{\alpha}^{n-3}\beta^{n-2} & \cdots & 0 & 0\\
\overline{\alpha}^{n-2}\beta^{n-1} & 0 & \cdots & 0 & 0\\
\end{pmatrix}\!,
\qquad \alpha, \beta \in \mathbb{C}\quad \textrm{and}\quad |\alpha|=|\beta|=1.
\]

We denote this matrix $W$ as $W_{\alpha, \beta}$,
and take $V=W_{1,1}$.
Observe that any $W_{\alpha, \beta}$ can be expressed as the product of $V$ and $U_{\alpha, \beta}$,
i.e.,
\[
W_{\alpha, \beta}=VU_{\alpha,\beta},
\]
therefore the gauge group $\mathcal{G}(\toep_n)$ is generated by $U_{\alpha, \beta}$ and $V$, with $|\alpha|=|\beta|=1$.
\end{proof}

Moreover,
if we denote by $\omega=\alpha\overline{\beta}$,
let
\[
\Omega_\omega
=
\begin{pmatrix}
1 & \overline{\omega} & \overline{\omega}^2 &\cdots & \overline{\omega}^{n-1}\\
\omega & 1 & \overline{\omega} & \cdots & \overline{\omega}^{n-2}\\
\omega^2 & \omega & 1 & \cdots & \overline{\omega}^{n-3}\\
\vdots & \vdots & \vdots & \ddots & \vdots\\
\omega^{n-1} & \omega^{n-2} & \omega^{n-3} & \cdots & 1\\
\end{pmatrix}\!,
\]
and we denote by $\Gamma \colon T\mapsto T^{\rm T}$ the transposition action of $T\in\toep_n$,
we obtain that
\begin{gather}
U_{\alpha, \beta}^*TU_{\alpha, \beta}=
\Omega_\omega \odot T,\label{eqs_grp_toep_1}
\\
U_{\alpha, \beta}^*V^*TVU_{\alpha, \beta}=
\Omega_\omega \odot \Gamma(T),\label{eqs_grp_toep_2}
\\
V^*U_{\alpha, \beta}^*TU_{\alpha, \beta}V=
\Omega_{\overline{\omega}} \odot \Gamma(T),\label{eqs_grp_toep_3}
\end{gather}
here $\Omega_\omega\odot T$ denotes the Schur product of $\Omega_\omega$ and $T$,
that is,
the elementwise product of $\Omega_\omega$ and~$T$.
Hence we obtain the following corollary.
\begin{Corollary}
The group of $\ucp_{\rk=1}(\toep_n)$ is isomorphic to the semidirect product of $U(1)$ and $\mathbb{Z}_2$,
and the gauge group $\gauge(\toep_n)$ is different from $\ucp_{\rk=1}(\toep_n)$ by a phase factor,
that is,
\begin{equation}\label{eq_ucp_toep}
\ucp_{\rk=1}(\toep_n)=U(1)\rtimes \mathbb{Z}_2
\end{equation}
and
\begin{equation}\label{eq_gauge_toep}
\gauge(\toep_n)=U(1)\times \left(U(1)\rtimes \mathbb{Z}_2\right).
\end{equation}
Moreover,
We have the short exact sequence which is independent of $n$:
\[
1\longrightarrow U(1) \longrightarrow \gauge(\toep_n) \longrightarrow \ucp_{\rk=1}(\toep_n) \longrightarrow 1.
\]
\end{Corollary}
\begin{proof}
We first show that the group $\ucp_{\rk=1}(\toep_n)$ is isomorphic to the semidirect product of $U(1)$ and $\mathbb{Z}_2$.
In fact,
according to the r.h.s.'s of equations \eqref{eqs_grp_toep_1}--\eqref{eqs_grp_toep_3},
the group $\ucp_{\rk=1}(\toep_n)$ is characterized by $\Omega_\omega$ and the transposition action $\Gamma$.
We observe that $\Gamma \circ \Omega_\omega\circ \Gamma=\Omega_{\overline{\omega}}$,
and if we equip the collection of matrices $\{\Omega_\omega\}_{\omega}$ with Schur product,
it is obvious to see that $\{\Omega_\omega\}_{\omega}$ is isomorphic to $U(1)$,
therefore we obtain the formula \eqref{eq_ucp_toep}.

Since $\omega$ is determined by the product of $\alpha$ and $\overline{\beta}$,
while the matrix $U_{\alpha, \beta}$ is determined by~$\alpha$ and~$\beta$,
hence compared with $\ucp_{\rk=1}(\toep_n)$,
the gauge group $\mathcal{G}(\toep_n)$ has one more $U(1)$-factor,
therefore we obtain the formula \eqref{eq_gauge_toep}.
\end{proof}

\begin{Remark}
Although the transposition map is not completely positive on $M_n(\mathbb{C})$,
however, it is unital completely positive on the Toeplitz system $\toep_n$ given by $V^*(\cdot) V$,
that is to say,
for a general $T\in M_n(\mathbb{C})$ we do not have $V^*TV=T^{\rm T}$,
while if $T\in \toep_n$ this equality does hold,
as is also discussed in \cite{MR4340832}.
\end{Remark}

\subsection{Perturbation semigroup of the Toeplitz system}\label{sec_pert_toep}
In this section, we shall characterize the semigroups $\pert(\toep_n)$ and $\pert^+(\toep_n)$.
We first need to recall the definition of the vectorization of a matrix as is defined in \cite{MR4383114}.

\begin{Definition}[{\cite[Section 2]{MR4383114}}]
Let $T\in M_{n\times m}(\mathbb{C})$,
we define the vectorization $\vct(T)$ of $T$ as
\[
\begin{aligned}
\vct\colon\ M_{n\times m}(\mathbb{C})
&\to\mathbb{C}^{nm},
\\
T&\mapsto
\sum_{j=1}^me_j^{(m)}\otimes Te_{j}^{(m)},
\end{aligned}
\]
here the tensor notation is in the sense of Kronecker product,
corresponding to the standard identification $\mathbb{C}^{nm}\cong \mathbb{C}^m\otimes \mathbb{C}^n$,
and $e_j^{(m)}$ denotes the $j-$th basis element in $\mathbb{C}^m$,
i.e.,
$e_j^{(m)}=(0,\ldots, 1,\ldots, 0)^{\rm T}$ with the $j$-th entry is equal to 1 and $0$'s otherwise.
\end{Definition}
For example,
if $T=(t_{ij})_{1\leq i,j\leq 3}\in M_3(\mathbb{C})$,
then
\[
\vct\colon\
\begin{pmatrix}
t_{11} & t_{12} & t_{13}\\
t_{21} & t_{22} & t_{23}\\
t_{31} & t_{32} & t_{33}
\end{pmatrix}
\mapsto
(t_{11}, t_{21}, t_{31}, t_{12}, t_{22}, t_{32}, t_{13}, t_{23}, t_{33})^{\rm T}.
\]

\begin{Remark}
As it is shown in \cite[Section 2]{MR4383114} we have the formula
\[
\vct\big(A\,X\,B^{\rm T}\big)=(B\otimes A)\vct(X),\qquad
A\in M_{n\times m}(\mathbb{C}),\quad
B\in M_{k\times l}(\mathbb{C}),\quad
X\in M_{m\times l}(\mathbb{C}).
\]
\end{Remark}

We take a matrix $\Delta\in M_{n^2\times (2n-1)}(\mathbb{C})$ as
\[
\Delta
=
\big( \vct(\tau_{-n+1}), \vct(\tau_{-n+2}),\ldots, \vct(\tau_{0}), \vct(\tau_{1}),\vct(\tau_{2}),\ldots, \vct(\tau_{n-1})\big).
\]
Consider the semigroup homomorphism $\Phi\colon \pert(\toep_n)\to \hu(\toep_n)$ as is defined in Section \ref{sec_pert}.
We denote the image of $\omega\in \pert(\toep_n)$ by $\varphi$,
i.e.,
$\varphi=\Phi(\omega)\in\hu(\toep_n)$.
Take $\{\tau_i\}_{-n+1\leq i\leq n-1}$ as the basis of $\toep_n$,
we can identify $\varphi$ with a $(2n-1)\times (2n-1)$ matrix $W=(w_{ij})_{-n+1\leq i,j\leq n-1}$ such that
\begin{equation}\label{eq_Phi_W}
\varphi(\tau_j)=\sum\limits_{i=-n+1}^{n-1}w_{ij}\tau_{i}.
\end{equation}

If we regard the tensor product in the definition of $\pert(\toep_n)$ as Kronecker product,
we can then treat an element $\omega\in\pert(\toep_n)$ as a $n^2\times n^2$ matrix,
which we still denote as $\omega$ without confusion.
In the case of Toeplitz operator system $\toep_n$,
the $C^*$-algebra generated by $\toep_n$ is $M_n(\mathbb{C})$.
The opposite algebra $M_n(\mathbb{C})^\circ$ is the transposition of $M_n(\mathbb{C})$,
and an element $a^\circ\in M_n(\mathbb{C})^\circ$ is just equal to $a^{\rm T}$.
The relationship between $\omega$ and $\varphi$ is described in the following proposition.
\begin{Proposition}\label{prop_toep_herm}
Let $\omega\in \pert(\toep_n)$,
then we have the equation
\begin{equation}\label{eq_omega_Delta}
\omega\Delta = \Delta \overline{W},
\end{equation}
here $W\in M_{2n-1}(\mathbb{C})$ is the square matrix associated with $\Phi(\omega)=\varphi\in\hu(\toep_n)$ defined by equation \eqref{eq_Phi_W},
and $\overline{W}$ denotes the elementwise complex conjugation of $W$.
\end{Proposition}

\begin{proof}
Let $\omega=\sum a_k\otimes b_k^{\rm T}\in \pert(\toep_n)$.
We observe that for $-n+1\leq j\leq n-1$,
the $j-$th column of $\omega\Delta$ is equal to
\[
\sum_i a_i\otimes b_i^{\rm T}(\vct(\tau_j))
=
\vct\bigg(\sum_i b_i^{\rm T}\tau_j a^{\rm T}_i\bigg)
=
\vct\bigg(\sum_i(a_i\tau_{-j}b_i)^{\rm T}\bigg)
=
\vct\big(\varphi(\tau_{-j})^{\rm T}\big).
\]
The equation \eqref{eq_Phi_W} implies that
\[
\vct\big(\varphi(\tau_{-j})^{\rm T}\big)=
\sum_{i=-n+1}^{n-1}w_{i,-j}\vct(\tau_{-i})
=\sum_{i=-n+1}^{n-1}w_{-i,-j}\vct(\tau_{i}).
\]
Since $\varphi$ is a Hermitian map,
we conclude that $w_{ij}=\overline{w_{-i,-j}}$.
Indeed, we observe that
\[
\varphi(\tau_j)=\varphi(\tau_{-j})^*
\Rightarrow
\sum w_{ij}\tau_i = \sum \overline{w_{-i,-j}}\tau_i
\Rightarrow
w_{ij}=\overline{w_{-i,-j}},
\]
hence
\begin{equation}\label{eq_j_column_pert}
\vct\big(\varphi(\tau_{-j})^{\rm T}\big)=
\sum_{i=-n+1}^{n-1}\overline{w_{ij}}\vct(\tau_i),
\end{equation}
notice that the l.h.s.\ of \eqref{eq_j_column_pert} is the $j$-th column of $\omega\Delta$,
and the r.h.s.\ of \eqref{eq_j_column_pert} is the $j$-th column of $\Delta\overline{W}$ for $-n+1\leq j \leq n-1$,
therefore we obtain the equation \eqref{eq_omega_Delta}.
\end{proof}

\begin{Remark}
To simplify the expression we count the rows and columns of the $(2n-1)\times (2n-1)$ matrix $W$ from $-n+1$ to $n-1$.
Since $\varphi$ is a unital map,
i.e.,
$\varphi(\tau_0)=\tau_0$,
the $0-$th column of~$W$ is $(0,\dots , 0,1,0,\dots , 0)^{\rm T}$ with $1$ in the central entry and $0$'s elsewhere.
\end{Remark}

\begin{Remark}
It is not difficult to show that $\rk(\Delta)=2n-1$ by a direct computation,
hence for each $\omega\in\pert(\toep_n)$ there is a unique $(2n-1)\times(2n-1)$ matrix $W$ satisfying the equation~\eqref{eq_omega_Delta}.
Especially,
we have that $\omega\Delta=\Delta$ if and only if $\Phi(\omega)=\id\in \ucp(\toep_n)$.
\end{Remark}

The matrix $\omega\in M_{n^2}(\mathbb{C})$ is not Hermitian in general.
However,
in \cite{MR599841} it is shown that we can transform $\omega$ to become a Hermitian matrix.

\begin{Definition}[{\cite[Section 1]{MR599841}}]
Let $T=(t_{ij})_{1\leq i, j\leq n^2}\in M_{n^2}(\mathbb{C})$,
we may write $T$ in the block form as $T=(T_{ij})_{1\leq i,j\leq n}$,
where $T_{ij}=(t_{rs}^{ij})_{1\leq r,s\leq n}\in M_n(\mathbb{C})$.
We define $\Gamma\colon M_{n^2}(\mathbb{C})\to M_n(M_n(\mathbb{C}))$ as follows:
\[
\Gamma(T)_{rs}^{ij}=t_{[i,j],[r,s]},\qquad
i,j,r,s = 1,\ldots, n,
\]
here $[i,j]=(i-1)n+j$.
\end{Definition}

That is to say,
we rearrange each row in $T\in M_{n^2}(\mathbb{C})$ to become a new block and then reorder all blocks together.
For example,
for n=2,
\[
T=
\begin{pmatrix}
t_{11} & t_{12} & t_{13} & t_{14}\\
t_{21} & t_{22} & t_{23} & t_{24}\\
t_{31} & t_{32} & t_{33} & t_{34}\\
t_{41} & t_{42} & t_{43} & t_{44}\\
\end{pmatrix}\!,
\qquad\textrm{and}\qquad
\Gamma(T)
=
\begin{pmatrix}
t_{11} & t_{12} & t_{21} & t_{22}\\
t_{13} & t_{14} & t_{23} & t_{24}\\
t_{31} & t_{32} & t_{41} & t_{42}\\
t_{33} & t_{34} & t_{43} & t_{44}\\
\end{pmatrix}\!.
\]

\begin{Theorem}[{\cite[Theorems 1 and~2]{MR599841}}]
Let $\mathcal{T}\colon M_n(\mathbb{C})\to M_n(\mathbb{C})$ be a linear map,
$\langle \mathcal{T}\rangle$ be the matrix representation of $\mathcal{T}$ with respect to the unit matrices $E_{i,j}$.
The following are equivalent:
\begin{itemize}
 \item $\mathcal{T}\colon M_n(\mathbb{C})\to M_n(\mathbb{C})$ is completely positive \emph(resp.~Hermitian-preserving\emph).
 \item There exist $A_1,\dots ,A_s\in M_n(\mathbb{C})$ such that $\langle \mathcal{T}\rangle=\sum_{i=1}^s A_i\otimes \overline{A}_i$ \emph(resp.~$\langle \mathcal{T}\rangle=\sum_{i=1}^s \epsilon_i A_i\otimes \overline{A}_i$ for $\epsilon_1,\dots ,\epsilon_s\in\{\pm 1\}$\emph).
 \item There exist $A_1, \dots , A_s\in M_n(\mathbb{C})$ and a $s\times s$ positive semidefinite \emph(resp.~Hermitian\emph) matrix~$(d_{ij})$ such that $\langle \mathcal{T}\rangle=\sum_{i,j=1}^sd_{ij}A_i\otimes \overline{A}_j$.
 \item $\Gamma(\langle \mathcal{T}\rangle)$ is positive semidefinite \emph(resp.~Hermitian\emph).
 \item $\Gamma\big(\langle \mathcal{T}\rangle^{\rm T}\big)$ is positive semidefinite \emph(resp.~Hermitian\emph).
\end{itemize}
\end{Theorem}
In our case,
we notice that if we regard $\omega$ as a matrix in $M_{n^2}(\mathbb{C})$,
then $\omega$ plays the role of~$\langle \mathcal{T}\rangle$ above.
Hence we have the following result:
\begin{Theorem}\label{thm_main}
If $\omega=\sum a_i\otimes b_i^\circ\in \pert(\toep_n)$ \emph(resp.~$\pert^+(\toep_n)$\emph),
then we have
\begin{itemize}
 \item $\Gamma(\omega)$ is a Hermitian \emph(resp.~positive semidefinite\emph) $n^2\times n^2$ matrix,
 \item $\varphi=\Phi(\omega)$ can be extended as a Hermitian-preserving \emph(resp.~completely positive\emph) map from $M_n(\mathbb{C})$ to $M_n(\mathbb{C})$,
 where $\Phi$ is defined as $\Phi(\omega)\colon X\mapsto \sum a_i (X) b_i$ for $X\in M_n(\mathbb{C})$,
 and hence we obtain the following two semigroup homomorphisms:
\begin{equation*}
\begin{aligned}
 \pert(\toep_n)\ \, &\xlongrightarrow{\Phi}\hu(\toep_n),
 \\
 \pert^+(\toep_n) &\xlongrightarrow{\Phi}\ucp(\toep_n).
\end{aligned}
\end{equation*}
\end{itemize}
\end{Theorem}

\begin{Example}
We now characterize the semigroup $\pert(\toep_2)$ and $\pert^+(\toep_2)$.
Since the basis of $\toep_2$ is $\{\tau_{-1}, \tau_{0},\tau_{1}\}$,
we take
\[
\Delta=(\vct(\tau_{-1}),\vct(\tau_{0}),\vct(\tau_{1}))
=
\begin{pmatrix}
0 & 1 & 0\\
1 & 0 & 0\\
0 & 0 & 1\\
0 & 1 & 0\\
\end{pmatrix}\!.
\]
Let $\varphi\in \hu(\toep_2)$,
then $\varphi$ is determined by a $3\times 3$ matrix
\[
W=\begin{pmatrix}
a & 0 & \overline{c}\\
b & 1 & \overline{b}\\
c & 0 & \overline{a}\\
\end{pmatrix}
\in M_3(\mathbb{C})
\]
given by equation \eqref{eq_Phi_W},
more explicitly,
\[
\begin{aligned}
\varphi\colon \ \toep_2\,\,\,\,
&\to\,\,\,
\toep_2,\\
\begin{pmatrix}
0 & 0\\
1 & 0
\end{pmatrix}
&\mapsto
\begin{pmatrix}
b & c \\
a & b
\end{pmatrix}\!,\\
\begin{pmatrix}
1 & 0\\
0 & 1
\end{pmatrix}
&\mapsto
\begin{pmatrix}
1 & 0 \\
0 & 1
\end{pmatrix}\!,\\
\begin{pmatrix}
0 & 1\\
0 & 0
\end{pmatrix}
&\mapsto
\begin{pmatrix}
\overline{b} & \overline{a} \\
\overline{c} & \overline{b}
\end{pmatrix}\!.
\end{aligned}
\]

Let
\[
T=\begin{pmatrix}
0 & 1 & 0 & 0\\
1 & 0 & 0 & 0\\
0 & 0 & 1 & 0\\
0 & 1 & 0 & 1\\
\end{pmatrix}\!,\qquad
I=\begin{pmatrix}
1 & 0 & 0\\
0 & 1 & 0\\
0 & 0 & 1\\
0 & 0 & 0\\
\end{pmatrix}\!,
\]
a direct calculation shows that $\Delta=T\,I$.
Let $\omega$ be an element in $\pert(\toep_n)$ such that $\Phi(\omega)=\varphi$,
the Proposition \ref{prop_toep_herm} implies that
\[
T^{-1}\,\omega\, T\, I=I\,\overline{W},
\]
thus $T^{-1}\,\omega\,T$ can be expressed as
\[
T^{-1}\,\omega\,T
=
\begin{pmatrix}
\overline{a} & 0 & c & z_1\\
\overline{b} & 1 & b & z_2\\
\overline{c} & 0 & a & z_3\\
0 & 0 & 0 & z_4\\
\end{pmatrix}
\]
for some $z_1,\ldots, z_4\in \mathbb{C}$,
and therefore
\[
\omega=
\begin{pmatrix}
 1-z_2 & \overline{b} & b & z_2 \\
 -z_1 & \overline{a} & c & z_1 \\
 -z_3 & \overline{c} & a & z_3 \\
 -z_2-z_4+1 & \overline{b} & b & z_2+z_4 \\
\end{pmatrix}\!,
\quad
\Gamma(\omega)=
\begin{pmatrix}
 1-z_2 & \overline{b} & -z_1 & \overline{a} \\
 b & z_2 & c & z_1 \\
 -z_3 & \overline{c} & -z_2-z_4+1 & \bar{b} \\
 a & z_3 & b & z_2+z_4 \\
\end{pmatrix}\!.
\]
According to Theorem \ref{thm_main},
 $\Gamma(\omega)$ is a Hermitian matrix;
thus we must have $z_2, z_4\in \mathbb{R}$ and $z_3=\overline{z_1}$.
Hence $\omega\in \pert(\toep_2)$ if and only if $\omega$ and $\Gamma(\omega)$ are of the forms
\[
\omega
=
\begin{pmatrix}
1-z_2 & \overline{b} & b & z_2\\
-z_1 & \overline{a} & c & z_1\\
-\overline{z}_1 & \overline{c} & a & \overline{z}_1\\
1-z_2-z_4 & \overline{b} & b & z_2+z_4\\
\end{pmatrix}\!,\qquad
\Gamma(\omega)
=
\begin{pmatrix}
 1-z_2 & \overline{b} & -z_1 & \overline{a} \\
 b & z_2 & c & z_1 \\
 -\overline{z}_1 & \overline{c} & -z_2-z_4+1 & \bar{b} \\
 a & \overline{z}_1 & b & z_2+z_4 \\
\end{pmatrix}
\]
with $z_2, z_4\in \mathbb{R}$ and $z_1\in \mathbb{C}$.
Moreover,
if $\Gamma(\omega)$ is positive semidefinite then $\omega\in \pert^+(\toep_n)$.
\end{Example}

We also obtain the positive definite matrix $\Gamma(\omega)$:
\[
\Gamma(\omega)
=
\begin{pmatrix}
1-z_2 & 0 & -z_1 & 1\\
0 & z_2 & 0 & z_1\\
-\overline{z_1} & 0 & 1-z_2-z_4 & 0\\
1 & \overline{z_1} & 0 & z_2+z_4\\
\end{pmatrix}\!.
\]

In the case of Toeplitz system,
since the $C^*$-algebra generated by $\toep_n$ is $M_n(\mathbb{C})$,
which is a nuclear $C^*$-algebra,
and since the Haagerup tensor norm is a $C^*$-cross norm \cite[Corollary 2.2]{MR891774},
we conclude that $\|\omega\|=\|\omega\|_h$ for an element $\omega\in \pert(\toep_n)$.
According to Proposition \ref{prop_ucp},
for $\omega\in \pert^+(\toep_n)$
we have $\|\omega\|=1$ .
We then obtain the following proposition.

\begin{Proposition}
Let $\varphi\in \hu(\toep_n)$,
$W\in M_{2n-1}(\mathbb{C})$ be the corresponding matrix,
and $\Delta=\big(\vct(\tau_{i})\big)_{-n+1\leq i\leq n-1}\in M_{n^2\times (2n-1)}(\mathbb{C})$.
A necessary condition for $\varphi\in \ucp(\toep_n)$ is that
$\big\|\Delta \overline{W}\big\|\leq \|\Delta\|$.
\end{Proposition}

\begin{proof}
If $\varphi\in \ucp(\toep_n)$,
i.e.,
the map $\toep_n \xlongrightarrow{\varphi}\toep_n$ is a UCP map,
according to Arveson's extension theorem~\cite{MR253059, MR1976867},
we can always extend $\varphi$ to a UCP map $\widetilde{\varphi}$ over $M_n(\mathbb{C})$,
i.e.,
$M_n(\mathbb{C})\xlongrightarrow{\widetilde{\varphi}}M_n(\mathbb{C})$,
and since any UCP map $\widetilde{\varphi}$ over $M_n(\mathbb{C})$ can be expressed as $\widetilde{\varphi}(X)=\sum V_i^* X V_i$ for finitely many $V_i\in M_n(\mathbb{C})$,
we can take $\omega=\sum V_i\otimes V_i^\circ$ such that $\Phi(\omega)=\varphi$.
By~Proposition \ref{prop_toep_herm} we have the equality $\omega \Delta=\Delta \overline{W}$.
Hence
\[
\big\|\Delta \overline{W}\big\|=\|\omega \Delta\|
\leq \|\omega\|\,\|\Delta\|,
\]
and since $\|\omega\|=1$,
we obtain that
$\big\|\Delta \overline{W}\big\|\leq \|\Delta\|.
$
\end{proof}

\appendix

\section{Operator systems}\label{apd_os}
This appendix contains some basic definitions and results about operator systems.
In our case we only consider the concrete operator systems,
i.e.,
$\mathcal{E}\subset B(\mathcal{H})$ for some Hilbert space $\mathcal{H}$.
We~refer the reader \cite{MR2111973, MR1793753, MR1976867} for more details about operator systems.
\begin{Definition}
Let $\mathcal{H}$ be a Hilbert space,
$B(\mathcal{H})$ be the set of all bounded operators on $\mathcal{H}$.
A concrete operator system is a (closed) linear subspace $\mathcal{E}$ of $B(\mathcal{H})$.
If $\mathcal{E}$ is closed under the involution,
i.e.,
$x\in \mathcal{E}$ implies $x^*\in \mathcal{E}$,
then $\mathcal{E}$ is called an operator system.
In this paper, we always assume the identity element $\id\in \mathcal{E}\subset B(\mathcal{H})$.
\end{Definition}
Let $\mathcal{H}^{(n)}$ be the direct sum of $n$ copies of $\mathcal{H}$,
$M_n(\mathcal{E})$ be the set of all $n\times n$ matrices with entries in $\mathcal{E}$.
Since we have the $C^*$-isomorphism $M_n(B(\mathcal{H}))\cong B\big(\mathcal{H}^{(n)}\big)$,
thus we can identify each element $(x_{ij})\in M_n(\mathcal{E})$ as an operator in $B\big(\mathcal{H}^{(n)}\big)$,
and $(x_{ij})$ inherits a norm $\|\cdot\|_n$ from $B\big(\mathcal{H}^{(n)}\big)$,
thus $M_n(\mathcal{E})$ turns out to be a normed vector space.

Let $\mathcal{E}\subset B(\mathcal{H})$ for be an operator system,
if there is a linear map $\varphi\colon \mathcal{E}\to\mathcal{E}$,
then we define $\varphi_n\colon M_n(\mathcal{E})\to M_n(\mathcal{E})$ by sending $(x_{ij})$ to $(\varphi(x_{ij}))$.
\begin{Definition}
Let $\mathcal{E}$ be an operator system,
$\varphi\colon \mathcal{E}\to \mathcal{E}$ be a linear map,
and $\varphi_n$ be the induced map $\varphi_n\colon M_n(\mathcal{E})\to M_n(\mathcal{E})$.
\begin{enumerate}
\item
The map $\varphi$ is called completely bounded if $\sup_{n>0}\|\varphi_n\| < \infty$,
and we set
\[
\|\varphi\|_{cb}=\sup_{n>0}\|\varphi_n\|.
\]
\item
The map $\varphi$ is called $n-$positive if $\varphi_n$ is positive,
and $\varphi$ is called completely positive if $\varphi_n$ is $n-$positive for all $n>0$.
\end{enumerate}
\end{Definition}

If a completely positive map $\varphi$ preserves the unit,
i.e.,
$\varphi(\id)=\id$,
then $\varphi$ is called a~UCP map(unital completely positive),
and we denote the collection of all UCP maps over $\mathcal{E}$ by $\ucp(\mathcal{E})$.

\begin{Theorem}[Arveson's extension theorem]
Let $\mathcal{A}$ be a $C^*$-algebra,
$\mathcal{E}$ an operator system contained in $\mathcal{A}$,
and $\varphi\colon \mathcal{E}\to B(\mathcal{H})$ a completely positive map.
Then there exists a completely positive map,
$\psi\colon \mathcal{A}\to B(\mathcal{H})$,
extending $\varphi$.
\end{Theorem}

According to Arveson's extension theorem we can always extend a map $\varphi\in \ucp(\mathcal{E})$ to a~map $\psi\in \ucp(B(\mathcal{H}))$.
In addition,
if $\psi$ is normal,
according to Kraus,
we can obtain a more explicit description of $\psi$.

\begin{Definition}
 We say a map $\psi\colon B(\mathcal{H})\to B(\mathcal{H})$ is normal if $\psi$ is ultraweakly continuous. Equivalently, for any trace class operator $T\in B_1(\mathcal{H})$,
take a sequence or more generally a~net $\{x_i\}_{i\in I}\subset B(\mathcal{H})$ and an $x\in B(\mathcal{H})$, if $\Tr(T\,x_i)\to \Tr(T\,x)$ then we have $\Tr(T\, \psi(x_i))\to \Tr(T\, \psi(x))$.
\end{Definition}

\begin{Theorem}[{\cite[Theorem 3.3]{MR292434}}]\label{thm_kraus}
Any linear mapping $T$ of $B(\mathcal{H})$ into itself with $\|TB\|\leq \|B\|$,
which is completely positive and ultraweakly continuous,
is of the form
\[
TB=\sum\limits_{k\in K} A_k^*BA_k \qquad \text{with} \quad
\sum\limits_{k\in K} A_k^* A_k \leq 1.
\]
\end{Theorem}

\begin{Theorem}[{\cite[Theorem 4.1]{MR292434}}]\label{thm_kraus_2}
Any completely positive ultraweakly continuous linear mapping $T$ of a von Neumann algebra $\mathfrak{U}$ into itself with $\|TB\| \leq \|B\|$ is of the form
\[
TB = \sum\limits_{k\in K} A_k ^*BA_k\qquad \text{with}\quad \sum\limits_{k\in K} A_k^* A_k \leq 1.
\]
\end{Theorem}

\begin{Remark}
In Theorems \ref{thm_kraus} and~\ref{thm_kraus_2} above,
the sum is in the sense of ultraweakly convergence for infinite $K$.
\end{Remark}

\section{Haagerup tensor product}\label{apd_haag}
In this appendix, we review some fundamental results about Haagerup tensor product of operator systems;
we refer to \cite{MR1793753, MR1976867, MR2006539} for more details.

Let $\mathcal{H}$ be a Hilbert space,
$B(\mathcal{H})$ the set of bounded operators over $\mathcal{H}$,
and let $\mathcal{E}$, $\mathcal{F}\subset B(\mathcal{H})$ be two operator systems.
We denote by $\mathcal{E}\otimes \mathcal{F}$ the space of algebraic tensor product,
i.e.,
\[
\mathcal{E}\otimes \mathcal{F}
=
\Bigg\{\sum_{i=1}^k a_i\otimes b_i\mid a_i\in \mathcal{E},\, b_i\in \mathcal{F},\, k\in \mathbb{N}\Bigg\}.
\]
We define the Haagerup tensor norm $\|x\|_h$ of $x\in \mathcal{E}\otimes \mathcal{F}$ as
\[
\|x\|_h:=\inf\Big\{\Big\|\sum a_ia_i^*\Big\|^{1/2}\Big\|\sum b_i^*b_i\Big\|^{1/2}\Big\},
\]
here the infimum runs over all the expressions of $x=\sum a_i\otimes b_i$.

\begin{Definition}
We denote by $\mathcal{E}\otimes_h\mathcal{F}$ the completion of $\mathcal{E}\otimes\mathcal{F}$ with respect to the Haagerup tensor norm $\|\cdot\|_h$.
\end{Definition}

\begin{Theorem}[{\cite[Theorem 17.4]{MR1976867}}]
Let $\mathcal{E}\subset\mathcal{E}_1$ and $\mathcal{F}\subset\mathcal{F}_1$ be operator systems.
Then the inclusion of $\mathcal{E}\otimes_h\mathcal{F}$ into $\mathcal{E}_1\otimes_h\mathcal{F}_1$ is a complete isometry.
\end{Theorem}

\begin{Theorem}[{\cite[Theorem 5.12]{MR2006539}}]
Let $\mathcal{A}\subset B(\mathcal{H})$ and $\mathcal{B}\subset B(\mathcal{K})$ be $C^*$-algebras.
We have a~natural completely isometric embedding
\[
J\colon\ \mathcal{A}\otimes_h\mathcal{B}\to\cb(B(\mathcal{K}, \mathcal{H}))
\]
defined by
\[
J(a\otimes b)\colon\ T\to aTb,
\]
here $\cb\left(B(\mathcal{K}, \mathcal{H})\right)$ denotes the collection of all the completely bounded maps over $B(\mathcal{K}, \mathcal{H})$.
\end{Theorem}
According to \cite{MR891774} the Haagerup tensor norm is a $C^*$-cross norm:
\begin{Theorem}[{\cite[Corollary 2.2]{MR891774}}]
Suppose $A$ and $B$ are $C^*$-algebras.
For any $a\in A$, $b\in B$, $\|a\otimes b\|_h=\|a\|\,\|b\|$.
\end{Theorem}

\subsection*{Acknowledgements}

The author wishes to express his gratitude to Walter van Suijlekom from Radboud University Nijmegen for stimulating discussions on related topics.
Besides that,
the author also would like to thank all the anonymous referees for their significant suggestions which improved the quality of this paper a lot.

\pdfbookmark[1]{References}{ref}
\LastPageEnding

\end{document}